\documentclass[letterpaper,final,10pt]{amsart}
\usepackage[utf8x]{inputenc}

\usepackage{amsmath}
\usepackage{amsfonts}
\usepackage{amssymb}
\usepackage{graphicx}
\usepackage{functan}
\usepackage{mathrsfs}
\usepackage[obeyspaces]{url}
\usepackage{braket}
\usepackage[all,cmtip,curve,knot]{xy}
\usepackage{tikz}
\usetikzlibrary{arrows}
\tikzstyle{block}=[draw opacity=0.7,line width=1.4cm]
\usepackage{algorithm}
\usepackage{algorithmicx}
\usepackage{algpseudocode}
\usepackage{mathbbol}
\usepackage{booktabs}
\usepackage{pdflscape}
\usepackage{wasysym}
\newtheorem{theorem}{Theorem}

\newtheorem{corollary}{Corollary}

\newtheorem{definition}{Definition}
\newtheorem{example}{Example}

\newtheorem{lemma}{Lemma}

\newtheorem{problem}{Problem}

\newtheorem{remark}{Remark}

\numberwithin{equation}{section}
\numberwithin{theorem}{section}
\numberwithin{lemma}{section}
\numberwithin{corollary}{section}
\numberwithin{definition}{section}
\numberwithin{example}{section}
\numberwithin{remark}{section}
\numberwithin{property}{section}
\numberwithin{proposition}{section}

\newcommand{\CU}[2]{C^*\left\langle #1 \left | #2 \right. \right\rangle}
\newcommand{\NS}[3][1]{#2_#1,\ldots,#2_#3}
\newcommand{\NCS}[1][\mathcal{R}]{\mathcal{Z}[#1]}
\newcommand{\EZ}[1][\mathcal{R}]{\mathcal{EZ}[#1]}
\newcommand{\BH}{\mathbb{B}(\mathcal{H})}
\newcommand{\ball}[2][1]{\mathbf{B}_#1(#2)}

\newcommand{\HS}{\mathcal{H}}
\newcommand{\lift}[3][\pi]{\xymatrix{#2 \ar@{~>}[r]_{#1} & #3}}

\newcommand{\Minf}{\mathcal{M}_\infty}

\newcommand{\Li}[2][]{\mathcal{L}_{#1}(#2)}
\newcommand{\ad}[2]{#1#2#1^*}
\newcommand{\Ad}[2]{\mathrm{Ad}[#1](#2)}
\newcommand{\diag}[2][]{\mathrm{diag}_{#1}\left[ #2 \right]}
\newcommand{\N}[1]{\mathcal{N}(#1)}
\newcommand{\U}[1]{\mathbb{U}(#1)}
\newcommand{\PU}[1]{\mathbb{PU}(#1)}

\newcommand{\I}{\mathbb{1}}

\newcommand{\RR}{\mathbb{R}}
\newcommand{\CC}{\mathbb{C}}
\newcommand{\ZZ}{\mathbb{Z}}

\newcommand{\num}[1]{\mathbf{N}_{#1}}

\newcommand{\dft}[1]{\mathscr{F}_{#1}}
\newcommand{\TT}[1][1]{{\mathbb{T}^{#1}}}

\newcommand{\stf}[2][\hat{u}]{\mathscr{D}_{\mathbb{T}}[#1](#2)}
\newcommand{\disk}[1][2]{{\mathbb{D}^{#1}}}

\newcommand{\stff}[1]{\mathscr{D}(#1)}
\newcommand{\Cliff}[1]{\mathrm{Cliff}(#1)}

\begin{document}
\title[Local matrix homotopies]{
Local Matrix Homotopies and Soft Tori}
\author{Terry A. Loring}
\author{Fredy Vides}
\address
{Department of Mathematics and Statistics \newline
\indent The University of New Mexico, Albuquerque, NM 87131, USA.}
\curraddr[F.~Vides]{\sc School of Mathematics and Computer Science \newline
\indent \sc National Autonomous University of Honduras, Ciudad Universitaria, 2do Piso, \newline
\indent Edificio F1, Honduras, C.A.}
\email[T.~A.~Loring]{loring@math.unm.edu}
\email[F.~Vides]{vides@math.unm.edu}



\date{\today}

\begin{abstract}
We present solutions to local connectivity problems in matrix representations of the form $C([-1,1]^{N}) \to A_{n,\varepsilon} \leftarrow C_{\varepsilon}(\mathbb{T}^{2})$
for any $\varepsilon\in[0,2]$ and any integer $n\geq 1$, where $A_{n,\varepsilon}\subseteq M_n$ and where $C_{\varepsilon}(\mathbb{T}^{2})$ denotes the {\bf Soft Torus}. We 
solve the connectivity problems by introducing the so called toroidal matrix 
links, which can be interpreted as normal contractive matrix analogies of free homotopies in differential algebraic topology.

In order to deal with the locality constraints, we have combined some techniques introduced in this document with 
some techniques from matrix geometry, combinatorial optimization, classification and representation theory of 
C$^*$-algebras.
\end{abstract}
\maketitle

\section{Introduction}\label{intro}

In this document we study the solvability of some local connectivity problems via constrained normal matrix homotopies in C$^*$-representations of the 
form
\begin{equation}
 C(\TT[N])\longrightarrow M_n,
 \label{first_representation}
\end{equation}
for a fixed but arbitrary integer $N\geq 1$ and any integer $n\geq 1$. In particular we study local normal matrix homotopies which preserve 
commutativity and also satisfy some additional constraints, like being rectifiable or piecewise analytic.

We build on some homotopic techniques introduced initially by Bratteli, Elliot, Evans and Kishimoto in \cite{Unitary_Homotopies_Bratteli} 
and generalized by Lin in \cite{Approximate_Homotopy_Lin} and \cite{Aproximate_diagonalization_Lin}. We combine the homotopic techniques 
with some techniques introduced here and some other 
 techniques from matrix geometry and noncommutative topology developed by Loring \cite{Loring_asymptotic_K_theory,LoringNCSets}, Shulman \cite{LoringNCSets}, 
 Bhatia \cite{Bhatia_mat_book}, 
Chu \cite{Chu_symmetric_homotopy}, 
Brockett \cite{Brockett_matching}, Choi \cite{Choi_Free,Choi_cp_lifting}, Effros \cite{Choi_cp_lifting}, Exel \cite{EilersExelSoftTorus}, Eilers \cite{EilersExelSoftTorus}, 
Elsner \cite{Elsner_Perturbation_theorems}, 
Pryde \cite{Pryde_Inequalities,Pryde_operator_equations}, McIntosh \cite{Pryde_operator_equations} 
and Ricker \cite{Pryde_operator_equations}, to construct the so called {\em toroidal matrix links}, which we use to obtain the main theorems presented in section 
\S\ref{main_results}, and which consist on 
local connectivity results in matrix representations of the form \ref{first_representation} and also of the form
\begin{equation}
 C(\TT[N])\longrightarrow M_n \longleftarrow C([-1,1]^{N}).
 \label{second_representation}
\end{equation}

Toroidal matrix links can be interpreted as noncommutative analogies of free homotopies in algebraic topology and topological deformation theory, 
they are introduced in section \S\ref{matrix_varieties_and_links} together with some other matrix geometrical objects.

In \S\ref{C_0_technique} we present a connectivity technique which provides us with very important imformation on the 
local uniform connectivity in matrix representations of the form $C(\TT[2])\to M_n$.

Given $\delta>0$, a function $\varepsilon:\RR\to \RR^+_0$ and two matrices $x,y$ in a set $S \subseteq M_n$ such that 
$\|x-y\|\leq \delta$, by a {\bf \em $\varepsilon(\delta)$-local matrix homotopy} between $x$ and $y$, we mean a 
matrix path $X\in C([0,1],M_n)$ such that $X_0=x$, $X_1=y$, $X_t\in S$ and $\|X_t-y\|\leq \varepsilon(\delta)$ for each $t\in [0,1]$. 
We write $x\rightsquigarrow_{\varepsilon} y$ to denote that there is a $\varepsilon$-local matrix homotopy betweeen $x$ and $y$.

The motivation and inspiration to study local normal matrix homotopies which preserve commutativity in C$^*$-representations of the form 
\ref{first_representation} and \ref{second_representation}, came from mathematical physics \cite[\S3]{Loring_Top_insulators} and matrix approximation theory 
\cite{Chu_num_lin}.

The problems from mathematical physics which motivated this study are inverse spectral problems, which consist on finding for a certain set of matrices 
$\NS{X}{N}$ which {\em approximately satisfy} a set of polynomial constraints $\mathcal{R}(x_1,\ldots,$ $x_N)$ on $N$ NC-variables, a set of {\em nearby matrices} 
$\NS{\tilde{X}}{N}$ which approximate $\NS{X}{N}$ and exactly satisfy the constraints $\mathcal{R}(\NS{x}{N})$. The problems from matrix approximation 
theory that we considered for this study, are of the type that can be reduced to the study of the solvability conditions for approximate and exact joint 
diagonalization problems for $N$-tuples of normal matrix contractions.

Since the problems which motivated the research reported in this document can be restated in terms of the study local piecewise analytic connectivity in matrix 
representations of the form $C_\varepsilon(\TT[2])\to M_n \leftarrow C(\TT[N])$ and $C_\varepsilon(\TT[2])\to M_n \leftarrow C([-1,1]^{N})$, 
we studied several variations of problems of the form. 

\begin{problem}[Lifted connectivity problem]
\label{first_problem}
 Given $\varepsilon>0$, is there $\delta>0$ such that the following conditions hold? For any integer $n\geq 1$, some prescribed sequence of linear compressions 
 $\kappa_{n}:M_{mn}\to M_n$ for some $m\geq 1$, and any two families of $N$ pairwise commuting normal contractions 
 $X_1,\ldots,X_N$ and $Y_1,\ldots,Y_N$ in $M_n$ which satisfy the constraints 
 $\|X_j-Y_j\|\leq \delta$, $1\leq j\leq N$, there are two families of $N$ pairwise commuting normal contractions 
 $\tilde{X}_1,\ldots,\tilde{X}_N$ and $\tilde{Y}_1,\ldots,\tilde{Y}_N$ in 
 $M_{mn}$ which satisfy the relations: $\kappa(\tilde{X}_j)=X_j$, $\kappa(\tilde{Y}_j)=Y_j$ and 
 $\|\tilde{X}_j-\tilde{Y}_j\|\leq \varepsilon$, 
 $1\leq j\leq N$. Moreover, there are $N$ peicewise 
 analytic $\varepsilon$-local homotopies of normal contractions $\mathbf{X}^1,\ldots,\mathbf{X}^N\in C([0,1],M_{mn})$ 
 between the corresponding pairs $\tilde{X}_j$, $\tilde{Y}_j$ in $M_{mn}$, which satisfy the relations 
 $\mathbf{X}_t^j \mathbf{X}_t^k= \mathbf{X}_t^k \mathbf{X}_t^j$, for each 
 $1\leq j,k\leq N$ and each $0\leq t\leq 1$.
 \end{problem} 

By solving problem {\bf P.}\ref{first_problem} we learned about the local connectivity of arbitrary $\delta$-close $N$-tuples of pairwise commuting normal 
contractions $\NS{X}{N}$ and $Y_1,\ldots,$ $Y_N$ in $M_n$, which was the main motivation of the research reported here. We also obtained some results concerning 
to the geometric structure of the joint spectra (in the sense of \cite{Pryde_operator_equations}) of the $N$-tuples. 

For a given $\delta>0$, the study of the solvability conditions of problems 
of the form {\bf P.}\ref{first_problem} provided us with geometric information about local deformations of particular representations of the form 
$C(\TT[N])\to A_0:=C^*(\NS{U}{N})\subseteq M_n$ and $C(\TT[N])\to A_1:=C^*(\NS{V}{N})\subseteq M_n$, where $\NS{U}{N},\NS{V}{N}\in \U{n}$ are pairwise 
commuting unitary matrices such that $\|U_j-V_j\|\leq \delta$. By local deformations we mean a family $\{A_t\}_{t\in[0,1]}\subseteq M_n$ of abelian C$^*$-algebras, with 
$A_t:=C^*(\mathbf{X}_t^1,\ldots,\mathbf{X}_t^N)$ and where $\mathbf{X}_t^1,\ldots,\mathbf{X}_t^N\in C([0,1],\U{n})$ are 
$\varepsilon(\delta)$-local matrix homotopies between $\NS{U}{N}$ and $\NS{V}{N}$ for some function $\varepsilon:\RR\to \RR^+_0$.

The main results are 
presented in \S\ref{main_results}, in section 
\S\ref{main_results_piecewise_analytic} we use toroidal matrix links to obtain some local piecewise analytic connectivity results which are 
non-uniform in dimension. In section \S\ref{main_lifted_results} we derive an uniform approximate connectivity technique via matrix homotopy lifting and 
in section \S\ref{compressible_matrix_sets} we present a connectivity lemma that can be used to derive some uniform connectivity results betweeen matrix 
representations of finite sets of universal algebraic contractions, some of the details of these constructions will be presented in \cite{Vides_matrix_words}.

\section{Preliminaries and Notation}
\label{notation}

\subsection{Matrix Sets and Operations}
Given two elements $x,y$ in a C$^*$-algebra $A$, we will write $[x,y]$ and $\Ad{x}{y}$ to denote the operations 
$[x,y]:=xy-yx$ and $\Ad{x}{y}:=\ad{x}{y}$.

Given any C$^*$-algebra $A$ and any element $x$ in $M_n(A)$, we will denote by $\diag[n]{x}$ the operation defined by the expression
\begin{eqnarray*}
 M_{n}(A)&\rightarrow& M_n(A)\\
 x&\mapsto&\diag[n]{x}\\
\left(
\begin{array}{cccc}
 x_{11} & x_{12} & \cdots & x_{1n}\\
 x_{21} & x_{22} & \cdots & x_{2n}\\
 \vdots & \vdots &  \ddots & \vdots \\
 x_{n1} & x_{n2} & \cdots & x_{nn}
\end{array}
\right)
&\mapsto& 
\left(
\begin{array}{cccc}
 x_{11} & 0 & \cdots & 0\\
 0 & x_{22} & \cdots & 0\\
 \vdots & \vdots &  \ddots & \vdots \\
 0 & 0 & \cdots & x_{nn}
\end{array}
\right).
\end{eqnarray*}

Given a C$^*$-algebra $A$, we write $\N{A}$, $\mathbb{H}(A)$ and $\U{A}$ to denote the sets of normal, hermitian and unitary elements 
in $A$ respectively. We will write 
$\N{n}$, $\mathbb{H}(n)$ and $\U{n}$ instead of $\N{M_n}$, $\mathbb{H}(M_n)$ and $\U{M_n}$. A 
normal element $u$ in a C$^*$-algebra $A$ is called a partial unitary if the element $uu^*=p$ is an otrhogonal projection in $A$, i.e. $p$ 
satisfies the relations $p=p^*=p^2$, we denote by $\PU{A}$ the set of partial unitaries in $A$ and we write $\PU{n}$ instead of 
$\PU{M_n}$. 

We write $\mathbb{I}$, $\mathbb{J}$, $\TT$ and $\disk$ to denote the sets $\mathbb{I}:=[0,1]$, $\mathbb{J}=[-1,1]$, 
$\TT:=\{z\in\CC|\: |z|=1\}$ and $\disk:=\{z\in\CC|\:|z|\leq 1\}$. For some arbitrary matrix  set $S\subseteq M_n$ and some arbitrary compact set 
$\mathbb{X}\subset \CC$, we will write $S(\mathbb{X})$ to denote the subset of elements in $S$ described by the expression,
\begin{equation}
 S(\mathbb{X}):=\{x\in S|\sigma(x)\subseteq \mathbb{X}\},
\end{equation}
for instance we can write $\N{n}(\disk)$ to denote the set of nomal contractions. We will denote by $\Minf$ 
the C$^*$-algebra described by 
\begin{equation}
 \Minf:=\overline{\bigcup_{n\in\ZZ^+}M_n}^{\|\cdot\|}.
\end{equation}

In this document we write $\I_n$ to denote the identity matrix in $M_n$. The symbol $\num{n}$ will be used to denote the diagonal matrices
\begin{equation}
 \num{n}:=\diag{n,n-1,\ldots,2,1}.
\end{equation}

We will write  
$\Omega_n$ and $\Sigma_n$ to denote the unitary matrices defined by
\begin{eqnarray}
 \Omega_n:=e^{\frac{2\pi i}{n}\num{n}}&=&\diag{1,e^{\frac{2\pi i (n-1)}{n}},\ldots,e^{\frac{4\pi i}{n}},e^{\frac{2\pi i}{n}}}
\end{eqnarray}
and
\begin{eqnarray}
 \Sigma_n&:=&
 \left(
 \begin{array}{cc}
  0 & \I_{n-1}\\
  1 & 0
 \end{array}
 \right).
\end{eqnarray}

\begin{remark}
 The unitary matrices $\Omega_n$ and $\Sigma_n$ are related, by the equation
\begin{eqnarray*}
 \Omega_n=\dft{n}^*\Sigma_n\dft{n},
 \label{CCR_cyclic_coord_unitaries}
\end{eqnarray*}
where $\dft{N}:=\left(\frac{1}{\sqrt{N}}e^{\frac{2\pi i (j-1)(k-1)}{N}}\right)_{1\leq j,k\leq N}$ is the discrete Fourier transform (DFT) unitary matrix.
\end{remark}

Given an abstract object (group or C$^*$-algebra) $A$ we write $A^{\ast N}$ to denote 
the operation consisting on taking the free product of $N$ copies of $A$.

\begin{definition}[Local preservers]
Given a linear mapping $K:M_N\to M_n$, with $n\leq N$, and given a set $S\subseteq M_n$, we say that $K$ {\bf locally preserves} 
$S$ if there is $T\subseteq M_N$ such that $K(T)\subseteq S$, if in particular $K(T)\subseteq \N{n}$ we say that $K$ {\bf locally 
preserves normality}.
\end{definition}

\begin{example}
 The linear compression $\kappa:M_{2n}\to M_n$ defined by
 \[
  \kappa:
  \left(
  \begin{array}{cc}
   x_{11} & x_{12}\\
   x_{21} & x_{22}
  \end{array}
  \right)
  \mapsto
  x_{11}
 \]
locally preserves normality with respect to the set $T:=\{X\in M_{2n}|x_{11}\in \N{n}\}$.
\end{example}

\begin{example}
 The linear map $\phi:M_n\to M_n,x\mapsto \mathbf{D}x$ with $n\geq 1$ and $\mathbf{D}=\frac{1}{n}\diag{1,\ldots,n}$, locally preserves commutativity with respect to the set 
 $C^*(\mathbf{D})$.
\end{example}

\subsection{Joint Spectral Variation}
\subsubsection{Clifford Operators}
Using the same notation as Pryde in \cite{Pryde_Inequalities}, let $\RR_{(N)}$ denote the Clifford algebra over $\RR$ 
with generators $\NS{e}{N}$ and relations $e_ie_j=-e_je_i$ for $i\neq j$ and $e^2_i=-1$. Then $\RR_{(N)}$ is an 
associative algebra of dimension $2^N$. Let $S(N)$ denote the set $\mathscr{P}(\{1,\ldots,N\})$. Then the elements 
$e_S=e_{s_1}\cdots e_{s_k}$ form a basis when $S=\{\NS{s}{k}\}$ and $1\leq s_1< \cdots<s_k\leq N$. Elements of 
$\RR_{(N)}$ are denoted by $\lambda=\sum_{S}\lambda_Se_S$ where $\lambda_S\in \RR$. Under the inner product 
$\scalprod*{\lambda}{\mu}=\sum_S \lambda_S \mu_S$, $\RR_{(N)}$ becomes a Hilbert space with orthonormal basis $\{e_S\}$.

The {\em Clifford operator} of $N$ elements $\NS{X}{N}\in M_n$ is the operator 
defined in  $M_n\otimes \RR_{(N)}$   by
\[
 \Cliff{\NS{X}{N}}:=i\sum_{j=1}^N X_j\otimes e_j.
\]
Each element $T=\sum_S T_S\otimes e_S\in M_n\otimes \RR_{(N)}$ acts on elements $x=\sum_S x_S\otimes e_S\in \CC^n\otimes \RR_{(N)}$ 
by $T(x):=\sum_{S,S'}T_s(x_{S'})\otimes e_Se_{S'}$. So $\Cliff{\NS{X}{N}}\in M_n\otimes \RR_{(N)}\subseteq \Li{\CC^n\otimes \RR_{(N)}}$. By 
$\|\Cliff{\NS{X}{N}}\|$ we will mean the operator norm of $\Cliff{\NS{X}{N}}$ as an element of $\Li{\CC^n\otimes \RR_{(N)}}$. As observed by 
Elsner in \cite[5.2]{Elsner_Perturbation_theorems} we have that
\begin{equation}
 \|\Cliff{\NS{X}{N}}\|\leq \sum_{j=1}^N\|X_j\|.
 \label{Clifford_bound}
\end{equation}

\subsubsection{Joint Spectral Matchings}

It is often convenient to have $N$-tuples (or $2N$-tuples) of matrices with real spectra. For this purpose we use the following construction, 
initiated by McIntosh and Pryde. If $X=(\NS{X}{N})$ is a $N$-tuple of $n$ by $n$ matrices then we can always decompose $X_j$ in the form 
$X_j=X_{1j}+iX_{2j}$ where the $X_{kj}$ all have real spectra. We write 
$\pi(X):=(X_{11},\ldots,X_{1N},X_{21},\ldots,X_{2N})$ and call $\pi(X)$ a partition of $X$. If the $X_{kj}$ all commute we say that 
$\pi(X)$ is a commuting partition, and if the $X_{kj}$ are simultaneously triangularizable $\pi(X)$ is a triangularizable partition. If 
the $X_{kj}$ are all semisimple (diagonalizable) then $\pi(X)$ is called a semisimple partition.

We say that $N$ normal matrices $\NS{X}{N}\in M_n$ are {\em simultaneously diagonalizable} if there is a unitary matrix 
$Q\in M_n$ such that 
$Q^* X_jQ$ is diagonal for each $j=1,\ldots,N$. In this case, for $1\leq k\leq n$, let 
$\Lambda^{(k)}(X_j):=(Q^*X_jQ)_{kk}$ the $(k,k)$ element of $Q^*X_jQ$, and set 
$\Lambda^{(k)}(\NS{X}{N}):=(\Lambda^{(k)}(X_1),\ldots,\Lambda^{(k)}(X_N))\in \CC^N$. The set
\[
 \Lambda(\NS{X}{N}):=\{\Lambda^{(k)}(\NS{X}{N})\}_{1\leq k\leq N}
\]
is called the joint spectrum of $\NS{X}{N}$. We will write $\Lambda(X_j)$ to denote the $j$-component of $\Lambda(\NS{X}{N})$, in other words 
we will have that
\[
 \Lambda(X_j)=\diag{\Lambda^{(1)}(X_j),\ldots,\Lambda^{(N)}(X_j)}.
\]

The following theorem was proved in McIntosh, Pryde and Ricker \cite{Pryde_operator_equations}.

\begin{theorem}[McIntosh, Pryde and Ricker]
\label{Pryde_Joint_spectral_variation}
Let $X=(\NS{X}{N})$ and $Y=(Y_1,\ldots,$ $Y_N)$ be $N$-tuples of commuting $n$ by $n$ normal matrices. There exists a 
permutation $\tau$ of the index 
set $\{1,\ldots,n\}$ such that
\begin{equation}
 \|\Lambda^{(k)}(\NS{X}{N})-\Lambda^{(\tau(k))}(\NS{Y}{N})\|\leq e_{N,0}\|\Cliff{X_1-Y_1,\ldots,X_N-Y_N}\|
\end{equation}
for all $k\in \{1,\ldots,n\}$.
\end{theorem}

In this theorem, $e_{N,0}$ is an explicit constant depending only on $N$ defined in 
\cite[(2.4)]{Pryde_operator_equations}.

\subsection{Amenable C$^*$-algebras and Bott elements}

The following lemma has been proved by H. Lin in \cite{Amenable_algebras_Lin}.

\begin{lemma}[H. Lin.]\label{Lin_Toroidal_Links_0}
 For any $\varepsilon>0$ and $d>0$, there exists $\delta>0$ satisfying the following: Suppose that $A$ is a unital C$^*$-algebra 
 and $u\in A$ is a unitary such that $\TT\backslash \sigma(u)$ contains an arc with length $d$. Suppose that $a\in A$ with $\|a\|\leq 1$ 
such that
\[
 \|ua-au\|<\delta.
\]
Then there is a self-adjoint element $h\in A$ such that $u=e^{ih}$,
\[
 \|ha-ah\|<\varepsilon \:\:\:\: \mathrm{and} \:\:\:\: \|e^{ith}a-ae^{ith}\|<\varepsilon
\]
for all $t\in\mathbb{I}$. If, furthermore, $a=p$ is a projection, we have
\[
 \left\|pup-p+\sum_{n=1}^\infty \frac{(iphp)^n}{n!}\right\|<\varepsilon.
\]

\end{lemma}

The following lemma was proved by H. Lin in \cite{Aproximate_diagonalization_Lin} using L.\ref{Lin_Toroidal_Links_0}, since for any integer 
$n\geq 1$ and any $u\in \U{n}$, we will have that $\TT\backslash \sigma(u)$ contains an arc with length at least $2\pi/n$.

\begin{lemma}[H. Lin.]
 \label{Lin_Toroidal_Links}
 Let $\varepsilon>0$, $n\geq 1$ be an integer and $M>0$. There exists $\delta>0$ satisfying the following: For any finite set 
 $\mathscr{F}\subset M_n$ with $\|a\|\leq M$ for all $a\in \mathscr{F}$, and a unitary $u\in M_n$ such that
 \[
  \|ua-au\|<\delta \:\:\:\: \mathrm{for \:\: all} \:\:\:\: a\in\mathscr{F},
 \]
there exists a continuous path of unitaries $\{u(t)\}_{t\in\mathbb{I}}\subset M_n$ with $u(0)=u$ and $u(1)=\I_n$ such that
\[
 \|u(t)a-au(t)\|<\varepsilon \:\:\:\: \mathrm{for \:\: all} \:\:\:\: a\in\mathscr{F}.
\]
Morover,
\[
 Length(\{u(t)\})\leq 2\pi.
\]
\end{lemma}

\begin{definition}({\bf The obstruction $Bott(u,v)$.})
\label{Loring_Bott_index_definition}
Given two unitaries in a $K_1$-simple real rank zero C$^*$-algebra $A$ that almost commute, the obstruction $Bott(u,v)$ is the 
Bott element associated to the two unitaries as defined by Loring in \cite{Loring_asymptotic_K_theory}. It is defined whenever $\|uv-vu\|\leq \nu_0$, 
where $\nu_0$ is a universal constant. It is defined as the $K_0$-class
\[
 Bott(u,v)=[\chi_{[1/2,\infty)}(e(u,v))]-
 \left[
 \left(
 \begin{array}{cc}
  1 & 0\\
  0 & 0
 \end{array}
 \right)
 \right],
\]
where $e(u,v)$ is a self-adjoint element of $M_2(A)$ of the form
\[
 e(u,v)=
 \left(
 \begin{array}{cc}
  f(v) & h(v)u+g(v)\\
  u^*h(v)+g(v) & 1-f(v)
 \end{array}
 \right),
\]
where $f$,$g$,$h$ are certain universal real-valued continuous functions on $\TT$.
\end{definition}

For details on the subject of $K$-theory for $C^*$-algebras the reader is referred to \cite{Rordam_book}. As observed by Bratteli, Elliot, 
Evans and Kishimoto in \cite{Unitary_Homotopies_Bratteli}, given a pair $u,v\in \U{A}$ we have that the obstruction $Bott(u,v)$ needs to vanish in order to be able to solve 
the problem $uvu^*\rightsquigarrow_{\varepsilon(\delta)} v$ by deforming $u\in \mathbb{U}_0(A)$ to $1$ continuously in $\U{A}$, when $\|uv-vu\|\leq \delta$.

\section{Matrix Varieties and Toroidal Matrix Links}
\label{matrix_varieties_and_links}
Let us denote by $\HS$ a universal separable Hilbert space, by $\BH$ the C$^*$-algebra of bounded operators on $\HS$, and 
for any given $S\subseteq \BH$ let us denote by $\ball[r]{S}$ the closed $r$-ball in $S$ defined by $\ball[r]{S}:=\{x\in S|\|x\|\leq r\}$.

Given some $N\in\ZZ^+$ and a set $\mathcal{R}(S)=\mathcal{R}(\NS{y}{N})$ of normed 
polynomial relations on the $N$-set $S:=\{\NS{y}{N}\}$ of NC-variables, we will call the set $\NCS$ 
described by
\begin{equation}
 \NCS:=\{\NS{x}{N}|\mathcal{R}(\NS{x}{N})\}
 \label{NC_Set_def}
\end{equation}
with $\NS{x}{N}\in \ball{\BH}$, a noncommutative semialgebraic set. 

\begin{example}
 As an example of normed $NC$-polynomial relations we can consider the set 
 $\mathcal{R}(x,y):=\{\|x^4-1\|\leq 10^{-10},\|y^7-1\|\leq 10^{-10},\|xy-yx\|\leq \frac{1}{8},xx^*=x^*x=1,yy^*=y^*y=1\}$.
\end{example}

Given a NC-semialgebraic set $\NCS$, we will use 
the symbol $\EZ$ to denote the universal C$^*$-algebra
\begin{equation}
 \EZ:=\CU{\NS{x}{N}}{\mathcal{R}(\NS{x}{N})},
 \label{environment_C_star_def}
\end{equation}
which we call the environment C$^*$-algebra of $\NCS$. For details on universal C$^*$-algebras described in terms of generators and relations 
the reader is referred to \cite{Loring_book}.

\begin{definition}[Semialgebraic Matrix Varieties]
\label{matrix_variety}
Given $J\in \ZZ^+$, a system 
of $J$ polynomials $\NS{p}{J}\in \Pi_{\braket{N}}=\mathbb{C}\Braket{\NS{x}{N}}$ in $N$ NC-variables $\NS{x}{N}\in \Pi_{\braket{N}}$ 
and a real number $\varepsilon\geq 0$, a particular matrix representation of the 
noncommutative semialgebraic set $\mathcal{Z}_{\varepsilon,n}(\NS{p}{J})$ 
described by 
\begin{equation}
 \mathcal{Z}_{\varepsilon,n}(\NS{p}{J}):=\Set{\NS{X}{N}\in M_{n} | \|p_j(\NS{X}{N})\|\leq \varepsilon, 1\leq j\leq J},
\end{equation}
will be called a {\bf $\varepsilon,n$-semialgebraic matrix variety} ($\varepsilon,n$-SMV), if $\varepsilon=0$ we can refer to the set as 
a {\bf matrix variety}. 
\end{definition}

\begin{example} As a first example, we will have that the 
set $\mathbf{Z_n}:=\{X\in M_n|\num{n}X$ $-X\num{n}=0\}$ is a matrix variety defined by the set with one $NC$-polynomial relation 
$\{\num{n}X-X\num{n}=0\}$. If for some $\delta>0$, we set now $\mathbf{Z}_{n,\delta}:=\{X\in M_n|\|[\num{n},X]\|\leq \delta\}$, the set $\mathbf{Z}_{n,\delta}$ is 
a matrix semialgebraic variety defined by the set with one normed $NC$-polynomial relation 
$\{\|\num{n}X-X\num{n}\|\leq \delta\}$.
\end{example}

\begin{example} Other example of a matrix semialgebraic variety, that has been useful to understand the geometric nature of the 
problems solved in this document, is described by the matrix set $\mathbf{Iso}_{\delta}(x,y)$, 
defined for some given $\delta\geq 0$ and any two normal contractions 
$x$ and $y$ in $M_n$, by the expression
\[
 \mathbf{Iso}_{\delta}(x,y):=
\left\{(z,w)\in\N{n}(\disk)\times\U{n}\left|\: 
\begin{array}{c}
 \|xw-wz\|=0 \\
 \|[z,y]\|=0\\
\|z-y\|\leq \delta
\end{array}
\right.
\right\}.
\]
\end{example}

\subsection{Toroidal Matrix Links}

\subsubsection{Finsler manifolds, matrix paths and toroidal matrix links}

\begin{definition}[Finsler manifold]
 A Finsler manifold is a pair $(M,F)$ where $M$ is a manifold and
$F : TM \to [0, \infty)$ is a function (called a Finsler norm) such that
\begin{itemize}
 \item $F$ is smooth on $TM\backslash \{0\}= \bigcup_{x\in M} \{T_x M \backslash \{0\}\}$,
 \item $F(v) \geq 0$ with equality if and only if $v = 0$,
 \item $F(\lambda v)$ = $\lambda F(v)$ for all $\lambda \geq 0$,
 \item $F(v + w) \leq F(v) + F(w)$ for all $w$ at the same tangent space with $v$.
\end{itemize}
\end{definition}

 Given a Finsler manifold $(M,F)$, the length of any rectifiable curve $\gamma : [a,b] \to M$ is given by the length functional
    \[L[\gamma] = \int_a^b F(\gamma(t),\partial_t \gamma(t))\,dt,\]
where $F(x, \cdot)$ is the Finsler norm on each tangent space $T_x M$.

The pair $(\mathcal{N},\|\cdot\|)$ is a Finsler manifold, where $\mathcal{N}$ denotes the set of normal matrices $\mathcal{N}$ (of any size) and 
$\|\cdot\|$ denotes the operator norm.

\begin{definition}[Matrix path curvature]
  Given a piecewise-$C^2$ matrix path $\gamma:[0,1]\to \mathcal{N}$, we define its curvature $\kappa[\gamma]$ to be
  \[
   \kappa[\gamma]:=\frac{1}{\|\partial_t\gamma(t)\|}\left\|\partial_t \left(\frac{\partial_t\gamma(t)}{\|\partial_t\gamma(t)\|}\right)\right\|.
  \]
 \end{definition}

\begin{definition}[Matrix flows] Given $n\geq 1$, a mapping $\phi:\mathbb{R^+_0}\times M_n \rightarrow M_n,$ $(t,x)\mapsto x_t$ will be called a 
matrix flow in this document. If we have in addition that $\sigma(x_t)=\sigma(x_s)$ for every $t,s\geq 0$, we say that the matrix 
flow is isospectral.
\end{definition}

\begin{definition}[interpolating path]
 Given two matrices $x$ and $y$ in $M_n$ and a matrix flow $\phi:\mathbb{I}\times M_n\rightarrow M_n$ such that $\phi_0(x)=x$ and $\phi_1(x)=y$, 
we say that the corresponding path $\{x_t\}_{t\in\mathbb{I}}:=\{\phi_t(x)\}_{t\in\mathbb{I}}\subseteq M_n$ is a solvent path for the interpolation 
problem $x\rightsquigarrow y$.
\end{definition}

\begin{definition}[$\circledast$ operation]
 Given two matrix paths $X,Y\in C([0,1],M_n)$ we write ${X\circledast Y}$ to denote the concatenation of $X$ and $Y$, which is 
 the matrix path defined in terms of $X$ and $Y$ by the expression,
 \[
  {X\circledast Y}_s:=
  \left\{
  \begin{array}{l}
   X_{2s},\:\: 0\leq s\leq \frac{1}{2},\\
   Y_{2s-1},\:\: \frac{1}{2}\leq s\leq 1.
  \end{array}
  \right.  
 \]
\end{definition}

\begin{definition}({\bf $\ell_{\|\cdot\|}$})
 Given a matrix path $\{x_t\}_{t\in\mathbb{I}}$ in $M_n$ we will write $\ell_{\|\cdot\|}(x_t)$ to denote the length of $\{x_t\}_{t\in\mathbb{I}}$ 
with respect to the operator norm which is defined by the expression
\[
 \ell_{\|\cdot\|}(x_t):=\sup \sum_{k=0}^{m-1}\|x_{t_{k+1}}-x_{t_{k}}\|,
\]
where the supremum is taken over all partitions of $\mathbb{I}$ as $0=t_0<\ldots<t_m=b$. If the function $x\in C(\mathbb{I},M_n)$ is a piecewise 
$C^1$ function, then
\[
 \ell_{\|\cdot\|}(x_t)=\int_{\mathbb{I}}\|\partial_tx_t\|dt.
\]
\end{definition}

\begin{definition}({\bf $\|\cdot\|$-flatness})
 A set $\mathcal{S}$ of $M_n$ is said to be $\|\cdot\|$-flat if any two points $x,y\in\mathcal{S}$ can be connected by 
a path $\{x_t\}_{t\in\mathbb{I}}\subseteq \mathcal{S}$ such that $\ell_{\|\cdot\|}(x_t)=\|x-y\|$.
\end{definition}

{\bf \begin{definition}[Toroidal matrix link] \label{matrix_links}
Given any two normal contractions $x,y$ in $M_n$, a toroidal matrix 
link is any piecewise analytic normal path $x_t:=\mathbb{K}[T_t(\mathbb{l}(x))]$ induced by a locally normal piecewise analytic 
matrix flow $T:\mathbb{I}\times M_{N}\rightarrow M_{N}$ 
with $N\geq n$, together with a locally normal compression $\mathbb{K}:M_N\rightarrow M_n$ with relative lifting map $\mathbb{l}:M_n\rightarrow M_N$, 
which satisfy the interpolating conditions $\mathbb{K}[T_0(\mathbb{l}(x))]=x$ 
and $\mathbb{K}[T_1(\mathbb{l}(x))]=y$ together with the constraints $\|\mathbb{K}[T_t(\mathbb{l}(x))]\|\leq 1$ for each $t\in\mathbb{I}$.
\end{definition}}

\begin{remark}
In the particular case where $[\mathbb{K}(T_t(\mathbb{l}(x))),\mathbb{K}(T_t(\mathbb{l}(y)))]=0$ for each $t\in\mathbb{I}$, whenever $[x,y]=0$, we call $T$ a 
toral matrix link.
\end{remark}

 \begin{remark} The {\em \bf curved nature} of the matrix varieties ({\em \bf as Finsler sub-manifolds of $\mathcal{N}$}) 
 whose local connectivity is studied in this document, induces an obstruction 
 to local connectivity via entirely {\em \bf flat} toroidal matrix links in general. The toroidal matrix links $\mathbf{T}\subset C([0,1],\mathcal{N})$ we have used to solve the connectivity problems which motivated 
 this study satisfy the constraint
  \[
   0\leq \kappa[T] \leq \frac{2}{\ell_{\|\cdot\|}(T)}, \:\: \forall T\in \mathbf{T}.
  \]
 \end{remark}

\subsection{Embedded matrix flows in solid tori}

Given some fixed but arbitrary $W\in \U{n}$, using the operation $\mathrm{diag}_n:M_n\rightarrow M_n$ 
one can define the mapping 
$\mathscr{D}:\U{n}\times M_n\rightarrow \disk$, determined by the expression.
\begin{eqnarray}
\U{n}\times M_n&\rightarrow& \disk \label{matrix_torus_definitions_1}\\
(W,x)&\mapsto&\stf[W]{x}\label{matrix_torus_definitions_2} \\
(W,x)&\mapsto&\{(\diag[n]{WxW^*})_{k,k}\}_{1\leq k\leq n}
\label{matrix_torus_definitions_3}
\end{eqnarray}
It is clear that $\diag{\stf[W]{x}}=\diag[n]{WxW^*}$ and that $\diag{\stf[\I_n]{x}}=\diag[n]{x}$, because of this when $W=\I_n$ we will 
write $\stff{x}$ instead of $\stf[\I_n]{x}$.

Given a matrix flow $\mathbb{I}\times \N{n}(\disk) \rightarrow \N{n}(\disk),(t,x)\mapsto X_t(x)$, one can identify $X$ with the set of 
flow lines in $\disk\times \TT$ determined by $\{(\stff{X_t(x)},e^{2\pi i t})\}_{t\in\mathbb{I}}$. The geometric picture determined by the mapping 
cylinder $\N{n}(\disk)\times \mathbb{I}\rightarrow \disk\times \TT,(x,t)\mapsto (\stff{X_t(x)},e^{2\pi i t})$ will be called the embedded matrix mapping 
cylinder relative to the flow $X$. We can think of the embedded matrix mapping cylinder in topological terms as a 
deformation described by 
the expression $\mathcal{D}_{X,Z_2}$, which is defined as
\begin{equation}
 \mathcal{D}_{X,Z_2}[Z_1\times \mathbb{I}]:= 
 \frac{(Z_1\times \mathbb{I}) \sqcup Z_2}{Z_1 \rightsquigarrow_{X_1} Z_2},
 \label{matrix_mapping_cylinder}
\end{equation}
where $Z_1$ and $Z_2$ are some prescribed matrix varieties such that $x\in Z_1$ and $X_1(x)\in Z_2$.

\begin{example}[Graphical example in $M_3$] \label{M3_graphical_example}
Let us set $\hat{u}_3:=e^{\frac{2\pi i}{3}f\left(\num{3}\right)}$ where $f\in C(\mathbb{I},\mathbb{I})$. For some prescribed $W_3\in \U{3}$, we can obtain a graphical example of 
a particular geometric picture of the computation of the embedded matrix mapping 
cylinder relative to the interpolating flow $\mathbf{U}$ which solves the problem $\hat{u}_3 \rightsquigarrow W_3\hat{u}_3W_3^*$. 

Let us set
\begin{eqnarray*}
 Z_1&:=&\{z\in\U{3}|[\hat{u}_3,z]=0\},\\
 Z_2&:=&\{z\in\U{3}|[W_3\hat{u}_3W_3^*,z]=0\}.
\end{eqnarray*}

Using projective methods, we can trace specific flow lines along the matrix flows corresponding to the 
dynamical deformation $\mathcal{D}_{\mathbf{U},Z_2}[Z_1\times \mathbb{I}]$, which solve the interpolation problem 
$\hat{u}_3 \rightsquigarrow W_3\hat{u}_3W_3^*$.

A particular (approximate) geometric picture of the matrix deformation induced by the toral matrix link $\{\mathbf{U}_t\}_{t\in\mathbb{I}}$ in $M_3$, projected in 
$\disk\times \TT$ for each $t\in\mathbb{I}$ via $\mathscr{D}_{\mathbb{T}}({\mathbf{U}_t})$ is presented in figures F.\ref{mapping_cylinder_1}-F.\ref{mapping_cylinder_3}.

\begin{figure}[!htb]
 \centering
\includegraphics[scale=0.28]{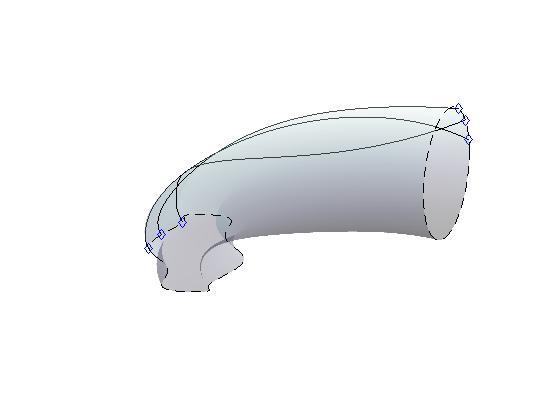}
\caption{Projected matrix mapping cylinder corresponding to the path $\mathbf{U}_{[0,\frac{1}{2}]}(\hat{u}_3)$ in $M_3$.}
\label{mapping_cylinder_1}
 \end{figure}
 \begin{figure}[!htp]
\centering
\includegraphics[scale=0.28]{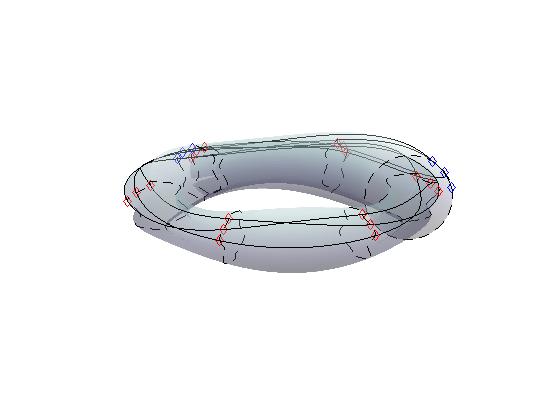}
\caption{Projected matrix mapping cylinder corresponding to the path $\mathbf{U}_{\mathbb{I}}(\hat{u}_3)$ in $M_3$.}
\label{mapping_cylinder_2}
 \end{figure}
\begin{figure}[!htb]
\centering
\includegraphics[scale=0.28]{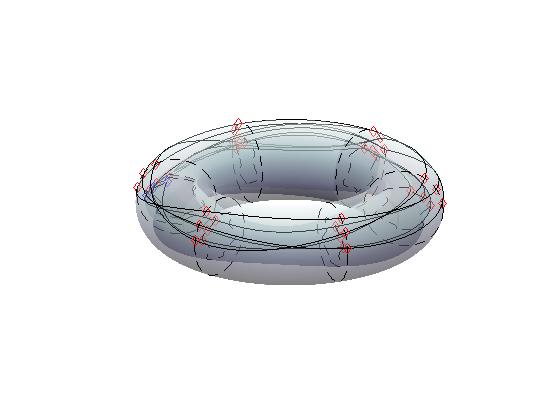}
\caption{Embedded matrix mapping cylinder corresponding to the path $\mathbf{U}_{\mathbb{I}}(\hat{u}_3)$ in $M_3$.}
\label{mapping_cylinder_3}
\end{figure}

Alternative methods to trace particular flow lines on mapping cylinders can be otained using matrix homotopies, this can be done using similar methods to 
the ones implemented in \cite{Chu_symmetric_homotopy}.
\end{example}

\subsection{Environment algebras}

\begin{definition}[Environment algebra (of a matrix algebra)]
Given a mat-rix algebra $A$ $\subseteq M_n$, a universal C$^*$-algebra 
$\mathcal{E}_A:=C^*_1\braket{\NS{x}{m}|\mathcal{R}(\NS{x}{m})}$ for which there is a matrix representation 
$\mathcal{E}_A\twoheadrightarrow \mathbf{E}_A \subseteq M_n$ such 
that $A\subseteq \mathbf{E}_A$, will be called an {\bf environment algebra} for $A$.
\end{definition}

Let us consider the universal C$^*$-algebras $C(\mathbb{J})$, $C(\TT)$, $C(\TT)\ast_{\CC} C(\TT)$, $C_{\delta}(\TT[2])$, $C_{\delta}(\mathbb{J}\times \TT)$ 
and $C^*_{\varepsilon}\langle\ZZ/2\times \ZZ \rangle$, defined in terms of 
generators and relations by the expressions.

\begin{eqnarray*}
  C(\mathbb{J}):=
C^*_1\left\langle u \left| 
                     \begin{array}{l}
                       h^*=h, \|h\|\leq 1
                      \end{array}
\right.\right\rangle
\end{eqnarray*}

\begin{eqnarray*}
  C(\TT):=
C^*_1\left\langle u \left| 
                     \begin{array}{l}
                       uu^*=u^*u=1
                      \end{array}
\right.\right\rangle
\end{eqnarray*}

\begin{eqnarray*}
  C(\TT)\ast_{\CC}C(\TT):=
C^*_1\left\langle u,v \left| 
                     \begin{array}{l}
                       uu^*=u^*u=1,\\
                       vv^*=v^*v=1
                      \end{array}
\right.\right\rangle
\end{eqnarray*}

\begin{eqnarray*}
 C_{\delta}(\TT[2]):=
C^*_1\left\langle u,v \left| 
                     \begin{array}{l}
                       uu^*=u^*u=1,\\
                       vv^*=v^*v=1,\\
                       \|uv-vu\|\leq \delta
                      \end{array}
\right.\right\rangle
\end{eqnarray*}

\begin{eqnarray*}
 C_{\delta}(\mathbb{J}\times \TT):=
C^*_1\left\langle h,u \left| 
                     \begin{array}{l}
                       h^*=h, \|h\|\leq 1\\
                       uu^*=u^*u=1,\\
                       \|hu-uh\|\leq \delta
                      \end{array}
\right.\right\rangle
\end{eqnarray*}

\[
 C^*_{\varepsilon}\langle\ZZ/2\times \ZZ \rangle:=
 C^*_1\left\langle u,v \left| 
                     \begin{array}{l}
                       uu^*=u^*u=u^2=1,\\
                       vv^*=v^*v=1,\\
                       \|uv-vu\|\leq \varepsilon
                      \end{array}
\right.\right\rangle.
\]

Let us consider now a local matrix representation result that we will use later in the construction of particular representation 
schemes.

\begin{lemma}
 \label{C_star_F2_inf_surj}
For every integer $n\geq 1$, there are $s_2,u_n,v_n\in\U{\Minf}$ such that the diagram
\[
 \xymatrix{
C(\TT)^{\ast2} \ar@{->>}[d] \ar@{->>}[r] & C^*\braket{(\mathbb{Z}/n)^{\ast 2}}  \ar@{->>}[r] & C^*_n(u_n,v_n) \ar@{=}[d]\\
C^*\braket{\mathbb{Z}/n\ast \mathbb{Z}/2} \ar@{->>}[r] & C^*_n(s_2,v_n) \ar@{=}[r] & M_n
}
\]
commutes, where $s_2\in \mathbb{H}(n)$, $u_n$ and $v_n$ are unitary elements in $M_n$.
\end{lemma}
\begin{proof}
Since we have that $C(\TT)^{\ast2}\simeq C^*\braket{\mathbb{F}_2}\simeq C^*(\ZZ^{\ast2})$, by universality of the 
C$^*$-representa-tions
\begin{eqnarray*}
  C^*(\ZZ^{\ast2})&\simeq&
C^*\left\langle u,v \left| 
                     \begin{array}{l}
                       uu^*=u^*u=\I,\\
                       vv^*=v^*v=\I
                      \end{array}
\right.\right\rangle \\
C^*((\ZZ/n)^{\ast2})&\simeq&
C^*\left\langle u,v \left| 
                     \begin{array}{l}
                       uu^*=u^*u=\I,\\
                       vv^*=v^*v=\I,\\
                       u^n=v^n=\I
                      \end{array}
\right.\right\rangle \\                      
 C^*(\ZZ/n\ast \ZZ/2)&\simeq&
C^*\left\langle u,v \left| 
                     \begin{array}{l}
                       uu^*=u^*u=\I,\\
                       vv^*=v^*v=\I,\\
                       u^n=v^2=\I
                      \end{array}
\right.\right\rangle,
\end{eqnarray*}
and by the structural properties of $M_n$, it is enough to find 
for any $n\in\ZZ^+$, up to unitary congruence in $M_n$, three unitaries $s_2,u_n,v_n\in\U{n}$ such that 
$C^*(s_2,v_n)= M_n= C^*(u_n,v_n)$ and $u_n^n=v_n^n=s_2^2=\I_n$, this 
can be done by taking for any $n\in\mathbb{Z}^+$ the orthogonal projection 
$p:=\diag{1,0,\ldots,0}\in\mathbb{H}(n)$ and the matrix $s_2=\I-2p\in \mathbb{H}(n)$, setting $u_n:=\Omega_n$
and $v_n:=\Sigma_n$ for $n\geq 2$ and $u_1=v_1=1$ for $n=1$, by functional calculus and direct computations it is easy to verify that $s_2,u_n,v_n\in\U{n}$ 
for every $n\in\mathbb{Z}^+$, and that $s_2=s_2^*$, it is also easy to verify 
that the system of matrix units 
$\set{e_{i,j,n}}_{1\leq i,j\leq n}$ and $u_n$ can be expressed as words in $C^*(s_2,v_n)$ for every $n\in\mathbb{Z}^+$, it is also 
clear that $p=e_{1,1,n}$ and hence, $s_2$ can be written as linear combinations of words in $C^*(u_n,v_n)$, we will then have that 
$C^*\braket{\ZZ/n\ast \ZZ/2}\twoheadrightarrow C^*(v_n,s_2)$ and 
$C^*\braket{\ZZ/n^{\ast2}}\twoheadrightarrow C^*(u_n,v_n)$ by the universal 
properties of $C^*\braket{\ZZ/2\ast\ZZ/n}$ and $C^*\braket{\ZZ/n^{\ast2}}$ respectively, since it can be easily verified that 
\begin{eqnarray*}
 u_n^n=v_n^n=s_2^2=\I_n,
\end{eqnarray*}
from these facts and the universal property of $C(\TT)^{\ast2}\simeq C^*\braket{\mathbb{F}_2}\simeq C^*\braket{\mathbb{Z}^{\ast2}}$, the result follows.
\end{proof}

\begin{remark}
\label{pre_localization}
 It can be seen that for any matrix C$^*$-subalgebra $A\subseteq M_n$, there is $\delta>0$ such that both $C(\TT)\ast_{\CC}C(\TT)$ and $C_{\delta}(\TT[2])$ are environment algebras of $A$. It can 
 also be seen that for any abelian C$^*$-subalgebra $D\subseteq M_n$, $C(\TT)$ is an environment algebra of $D$.
\end{remark}

\section{Local Matrix Connectivity}

\label{main_results}

\subsection{Topologically controlled linear algebra and Soft Tori}

\begin{definition}[Controlled sets of matrix functions]
\label{nc_controlled_topology}
 Given $\delta>0$, a function $\varepsilon:\RR\to\RR^+_0$, a finite set of functions $F \subseteq C(\TT,\disk)$ and two unitary matrices $u,v\in M_n$ such that 
 $\|uv-vu\|\leq \delta$, we say that the set 
$F$ is $\delta$-controlled by $\mathrm{Ad}[v]$ if the diagram,
\[
 \xymatrix{
C^*(u,v) & C^*(u) \ar[l] \ar[d]_{\mathrm{Ad}[v]}  & \{u\} \ar[d]_{\mathrm{Ad}[v]} \ar[l]^{\imath} \ar[rd]^{f}_{\approx_{\varepsilon(\delta)}} &  \\
& C^*(vuv^*) \ar[lu]  & \{vuv^*\} \ar[l]^{\imath} \ar[r]_{f} & \N{n}(\disk)
}
\]
commutes up to an error $\varepsilon(\delta)$ for each $f\in F$.
\end{definition}

\begin{remark}
 The $C^*$-homomorphism $C_{\delta}(\TT[2])\to C^*(u,v)$ allows us to see that the Soft Torus $C_{\delta}(\TT[2])$ provides an environment algebra for any $\delta$-controlled set of matrix functions.
\end{remark}

\begin{lemma}[Existence of isospectral approximants]
\label{Joint_spectral_variation_inequality_2}
 Given $\varepsilon>0$ there is $\delta> 0$ such that, for any $2$ families of $N$ pairwise commuting normal 
 matrices $\NS{x}{N}$ and $\NS{y}{N}$ 
 which satisfy the constraints $\|x_j-y_j\|\leq \delta$ for each $1\leq j\leq N$, there is a $C^*$-homomorphism  
 $\Psi$ such that $\sigma(\Psi(x_j))=\sigma(x_j)$, $[\Psi(x_j),y_j]=0$ and 
 $\max\{\|\Psi(x_j)-y_j\|,\|\Psi(x_j)-x_j\|\}\leq \varepsilon$, for each $1\leq j\leq N$.
\end{lemma}
\begin{proof}
  By changing basis if necessary, we can assume that $\NS{y}{N}$ are diagonal matrices. From T.\ref{Pryde_Joint_spectral_variation} we will have that there is a permutation $\tau$ of 
 the index set $\{1,\ldots,n\}$ such that for each $1\leq k\leq n$ we have that
 \begin{eqnarray}
  |\Lambda^{(k)}(x_j)-\Lambda^{(\tau(k))}(y_j)|&\leq& 
  \|\Lambda^{(k)}(\NS{x}{N})-\Lambda^{(\tau(k))}(\NS{y}{N})\| \nonumber \\
  &\leq& e_{N,0}\|\Cliff{x_1-y_1,\ldots,x_N-y_N}\|.
\label{Joint_spectral_variation_inequality}
  \end{eqnarray}
 Using \ref{Clifford_bound} and as a consequence of \ref{Joint_spectral_variation_inequality} we can find a permutation matrix $\mathcal{T}\in \U{n}$ such that
 \begin{eqnarray}
  \|\mathcal{T}^*\diag{\Lambda(x_j)}\mathcal{T}-\diag{\Lambda(y_j)}\|&\leq& e_{N,0}\|\Cliff{x_1-y_1,\ldots,x_N-y_N}\|\nonumber \\
                                                                     &\leq& e_{N,0}N\delta, \:\: 1\leq j\leq N.
  \label{Existence_of_nu_reaching_Joint_diagonalizer}
 \end{eqnarray}
 Let us set $c_N:=e_{N,0}N$. For the matrices $\NS{x}{N}$ there is a unitary joint diagonalizer $W\in M_n$ such that 
 $W\diag{\Lambda(x_j)}W^*=x_j$, $1\leq j\leq N$,
 \begin{eqnarray}
   \|W\diag{\Lambda(x_j)}W^*-\mathcal{T}^*\diag{\Lambda(x_j)}\mathcal{T}\|
   &\leq&\|W\diag{\Lambda(x_j)}W^*-y_j\|\nonumber \\ 
   &&+\|y_j-\mathcal{T}^*\diag{\Lambda(x_j)}\mathcal{T}\| \nonumber\\
   &\leq&(1+c_N)\|x_j-y_j\|\leq (1+c_N)\delta.
   \label{controlled_JMP_inequalities}
 \end{eqnarray}
 If we set $V:=W\mathcal{T}$ and $\varepsilon=(1+c_N)\delta$, we will have that by \ref{Existence_of_nu_reaching_Joint_diagonalizer} and 
 \ref{controlled_JMP_inequalities} the inner $C^*$-automorphism $\Psi:=\mathrm{Ad}[V^*]$  satisfies the constraints in the statement of 
 this lemma, and we are done.
\end{proof}

\begin{remark}
 \label{isospectral_approximants_remark}
 The $C^*$-automorphism $\Psi$ from L.\ref{Joint_spectral_variation_inequality_2} is called an isospectral approximant for the two $N$-tuples 
$\NS{x}{N}$ and $\NS{y}{N}$. If $\Psi:=\mathrm{Ad}[W^*]$ for some $W\in \U{n}$, then we will have that its inverse $\Psi^{\dagger}$ will be given 
by the expression $\Psi^{\dagger}=\mathrm{Ad}[W]$.
\end{remark}

\begin{remark}
 The constant $c_N$ in the proof of L.\ref{Joint_spectral_variation_inequality_2} depends only on the number $N$ of matrices in each family. It does not depend on 
 the matrix size.
\end{remark}

\subsection{Local piecewise analytic connectivity}

\label{main_results_piecewise_analytic}

In this section we will present some piecewise analytic local connectivity results in matrix representations 
of the form $C_\varepsilon(\TT[2])\to M_n \leftarrow C(\TT[N])$ and 
$C_\varepsilon(\mathbb{J}\times \TT)\to M_n \leftarrow C(\mathbb{J}^N)$. 

\begin{theorem}[Local normal toral connectivity]
\label{local_connection_of_N_tuples_of_normal_contractions}
 Given $\varepsilon>0$ and any $n\in\ZZ^+$, there is $\delta>0$ such that, for any $2N$ normal contractions $\NS{x}{N}$ and $y_1,\ldots,$ $y_N$ in $M_n$ which 
satisfy the relations
\[
\left\{
\begin{array}{l}
 [x_j,x_k]=[y_j,y_k]=0, \:\:\: 1\leq j,k\leq N,\\
 \|x_j-y_j\|\leq \delta, \:\:\: 1\leq j\leq N,
\end{array}
\right.
\]
there exist $N$ toral matrix 
links $X^1,\ldots,X^N$ in $M_n$, which solve the problems
\[
 x_j \rightsquigarrow y_j, \:\:\: 1\leq j\leq N,
\]
and satisfy the constraints
\[
\left\{
\begin{array}{l}
 \|X_t^j(x_j)-y_j\|\leq \varepsilon,
\end{array}
\right.
\]
for each $1\leq j,k\leq N$ and each $t\in \mathbb{I}$. Moreover, $\ell_{\|\cdot\|}(X_t^j(x_j))\leq \varepsilon, \:\:\: 1\leq j\leq N$.
\end{theorem}
\begin{proof}

By L.\ref{Lin_Toroidal_Links_0}, L.\ref{Lin_Toroidal_Links} and L.\ref{Joint_spectral_variation_inequality_2} we will have that given $\varepsilon>0$, 
 there are $0<\delta\leq \nu\leq \varepsilon/2$ and an isospectral approximant $\Psi:=\mathrm{Ad}[W^*]$ (with $W\in\U{n}$) for $\NS{x}{N}$ and $\NS{y}{N}$ such that, 
 $\max\{\|x_j-\Psi(x_j)\|,\|y_j-\Psi(x_j)\|\}\leq \nu$ and $[\Psi(x_j),y_j]=0$ for each $1\leq j\leq N$, we will also have that there is a unitary path $\mathcal{W}\in C(\mathbb{I},M_n)$ which is defined by the 
 expression $\mathcal{W}_t:=e^{-itH_W}$ for each $t\in\mathbb{I}$, where $H_W\in M_n$ is a hermitian matrix such that $e^{i H_W}=W$ and 
 $\|[H_W,x_j]\|\leq \varepsilon/2$ for each $1\leq j\leq N$, and that is defined by $H_W:=h(W)$, for some function 
$h:\Omega^{\alpha}_{d,s}\rightarrow [-1,1]$, and where $\sigma(W)\subset \Omega^{\alpha}_{d,s}:=\{e^{i(\pi t+\alpha)}|-1+s<t<1-s\}\subset \TT$, with $s,\alpha\in\RR$ chosen in such a 
way that $\TT\backslash \Omega^{\alpha}_{d,s}$ contains an arc of length $d$ (with $d\geq 2\pi/n$). Moreover, we can choose $\delta$ and $\nu$ in such a way that 
the path $\mathcal{W}$ satisfies the inequalities $\|[\mathcal{W}_t,\Psi(x_j)]\|\leq \varepsilon/2$ for each $t\in [0,1]$ and each $1\leq j\leq N$. 

It can be seen that the paths $\breve{X}^j_t:= \Ad{\mathcal{W}_t}{x_j}$ will solve the problem $x_j \rightsquigarrow_{\varepsilon/2} \Psi(x_j)$ for each $1\leq j\leq N$. Let 
us set $\bar{X}_t^j:=(1-t)\Psi(x_j)+ty_j$, we can now construct $N$ 
toroidal matrix links of the form $X^j:=\breve{X}^j\circledast \bar{X}^j$ which solve the problems 
$x_j\rightsquigarrow y_j$, locally preserve normality and commutativity and satisfy the $\|\cdot\|$-distance constraints
\begin{eqnarray*}
 \|X^j_t-y_j\|&\leq&\|X^j_t-\Psi({x}_j)\|+\|y_j-\Psi({x}_j)\|\\ 
               &\leq& \frac{\varepsilon}{2}+\nu\\
          &\leq& \frac{\varepsilon}{2} + \frac{\varepsilon}{2} = \varepsilon, 
\end{eqnarray*}
together with the $\|\cdot\|$-length constraints
 \begin{eqnarray}
  \ell_{\|\cdot\|}(X^j_t)&\leq&\ell_{\|\cdot\|}(\breve{X}^j_t)+\|\Psi({x}_j)-y_j\|\\
                         &=&\int_{\mathbb{I}}\|\partial_t \mathrm{Ad}[\mathcal{W}_t](x_j)\|dt+\|\Psi({x}_j)-y_j\|\\
                         &=&\|[H_W,\Psi({x}_j)]\|+\|\Psi({x}_j)-y_j\|\\
                         &\leq&\frac{\varepsilon}{2}+\nu\leq \varepsilon,
 \end{eqnarray}
which hold whenever $\|x_j-y_j\|\leq \delta$, $1\leq j\leq N$, and we are done.
\end{proof}

\begin{remark}
 It can be noticed that the solvent matrix links $X^1,\ldots$ $,X^N $ whose existence is stated 
 in T.\ref{local_connection_of_N_tuples_of_normal_contractions} are factored in the form $X^j=\breve{X}^j\circledast \bar{X}^j$, 
 we call $\breve{X}^j$ and $\bar{X}^j$ the {\em \bf curved} and {\em \bf flat} factors of $X^j$ respectively.
\end{remark}

We will derive now, some corollaries of the proof of T.\ref{local_connection_of_N_tuples_of_normal_contractions}.

\begin{corollary}[Local hermitian toral connectivity]
\label{local_hermitian_interpolants}
  Given $\varepsilon>0$ and any integer 
  $n\geq 1$, there is $\delta>0$ such that, for any $2N$ hermitian contractions $\NS{x}{N}$ and $\NS{y}{N}$ in $M_n$ which 
satisfy the relations
\[
\left\{
\begin{array}{l}
 [x_j,x_k]=[y_j,y_k]=0, \:\:\: 1\leq j,k\leq N,\\
 \|x_j-y_j\|\leq \delta, \:\:\: 1\leq j\leq N,
\end{array}
\right.
\]
there exist $N$ toral matrix 
links $X^1,\ldots,X^N$ in $M_n$, which solve the problems
\[
 x_j \rightsquigarrow y_j, \:\:\: 1\leq j\leq N,
\]
and satisfy the constraints
\[
\left\{
\begin{array}{l}
 X_t^j(x_j)=(X_t^j(x_j))^*,\\
 \|X_t^j(x_j)-y_j\|\leq \varepsilon,
\end{array}
\right.
\]
for each $1\leq j,k\leq N$ and each $t\in \mathbb{I}$. Moreover, $\ell_{\|\cdot\|}(X_t^j(x_j))\leq \varepsilon, \:\:\: 1\leq j\leq N$.
\end{corollary}
\begin{proof}
 Since for any $\alpha\in\RR$, any pair of hermitian matrices $x,y\in \mathbb{H}(n)$ and any partial unitary $z\in \PU{n}$, we have that $x+\alpha(y-x)$ and 
 $zxz^*$ are also in $\mathbb{H}(n)$, the result follows as a consequence of L.\ref{Joint_spectral_variation_inequality_2} and 
 T.\ref{local_connection_of_N_tuples_of_normal_contractions}.
\end{proof}

\begin{corollary}[Local unitary toral connectivity]
\label{local_connectivity_N_tuples_commuting_unitaries}
 Given any $\varepsilon\geq 0$ and any integer $n\geq 1$, there is $\delta\geq 0$ such that given any $2N$  
 unitary matrices $U_{1},\ldots,U_{N}$ ,$V_{1},\ldots,V_{N}$ in $M_{n}$ which satisfy the relations 
  \begin{eqnarray*}
  \left\{
  \begin{array}{l}
   [U_{j},U_{k}]=[V_{j},V_{k}]=0,\\
   \|U_{k}-V_{k}\|\leq \delta,
  \end{array}
 \right.
 \end{eqnarray*}
for each $1\leq j,k \leq N$, there are toral matrix links $u^1,\ldots,u^N$ in $M_{n}$ which solve the interpolation problems 
 \[
 U_{k}\rightsquigarrow V_{k}, \:\:\: 1\leq k\leq N,
 \]
 and also satisfy the relations 
 \begin{eqnarray*}
  \left\{
  \begin{array}{l}
   (u^j_{t})^*u^j_{t}=u^j_{t}(u^j_{t})^*=\I_n,\\
   \|u^j_t-V_j\|\leq \varepsilon,
  \end{array}
 \right.
 \end{eqnarray*}
for each $t\in\mathbb{I}$ and each $1\leq j,k\leq N$. Moreover, $\ell_{\|\cdot\|}(u^j_{t})\leq \varepsilon$, $1\leq j\leq N$.
\end{corollary}
\begin{proof}
 Since for any $C^*$-automorphisms $\Psi$ we have that $\Psi(\U{n})\subseteq \U{n}$, and since any two commuting unitaries 
 $U$ and $V$ can connected by a flat unitary path $\bar{U}_t:=Ue^{t\ln\left(U^{*}V\right)}$, for $0\leq t\leq 0$. We will 
 have that the result can be derived using a similar argument to the one implemented in the proof of 
 T.\ref{local_connection_of_N_tuples_of_normal_contractions}.
 \end{proof}

\subsubsection{Lifted local piecewise analytic connectivity}

\label{main_lifted_results}

Let us denote by $\kappa$ the matrix compression $M_{2n}\to M_n$ defined by the mapping
 \begin{eqnarray*}
  \kappa:M_{2n}\to M_n,
        \left(
        \begin{array}{cc}
         x_{11} & x_{12} \\
         x_{21} & x_{22}
        \end{array}
        \right)
        \mapsto x_{11}.    
 \end{eqnarray*}
Let us write $\imath_2:M_n\to M_{2n}$ to denote the $C^*$-homomorphism defined by the expression $\imath_2(x):=x\oplus x=\I_2\otimes x$.

\begin{definition}[Standard dilations]
\label{Std_dilations}
 Given a $C^*$-automorphism $\Psi:=\mathrm{Ad}[W]$ (with $W\in \U{n}$) in $M_n$, we will denote by 
$\Psi^{[s]}$ the $C^*$-automorphism in $M_{2n}$ defined by the expression 
$\Psi^{[s]}:=\mathrm{Ad}[\I_2\otimes W]=\mathrm{Ad}[W\oplus W]$. We call $\Psi^{[s]}$ a standard dilation 
of $\Psi$.
\end{definition}

\begin{definition}[$\ZZ/2$-dilations]
\label{Z2_dilations}
Given a $C^*$-automorphism $\Psi:=\mathrm{Ad}[W]$ (with $W\in \U{n}$) in $M_n$, we will denote by 
$\Psi^{[2]}$ the $C^*$-automorphism in $M_{2n}$ defined by the expression 
$\Psi^{[2]}:=\mathrm{Ad}[(\Sigma_2\otimes \I_n)(W^*\oplus W)]$. We call $\Psi^{[2]}$ a $\ZZ/2$-dilation 
of $\Psi$.
\end{definition}

\begin{remark}
\label{local_matching_remark}
 It can be seen that $\kappa(\imath_2(x))=x$ for any $x\in M_{2n}$, it can also be seen that 
 $\kappa(\Psi^{[2]}(\imath_2(x)))=\kappa(\Psi^{[s]}(\imath_2(x)))$.
\end{remark}

\begin{theorem}[Lifted local toral connectivity]
 \label{lifted_connectivity}
\label{lifted_local_connection_of_N_tuples_of_normal_contractions}
 Given $\varepsilon>0$, there is $\delta>0$ such that, for any $2N$ normal contractions $\NS{x}{N}$ and $\NS{y}{N}$ in $M_n$ which 
satisfy the relations
\[
\left\{
\begin{array}{l}
 [x_j,x_k]=[y_j,y_k]=0, \:\:\: 1\leq j,k\leq N,\\
 \|x_j-y_j\|\leq \delta, \:\:\: 1\leq j\leq N,
\end{array}
\right.
\]
there is a $C^*$-homomorphism $\Phi:M_n\to M_{2n}$ and $N$ toral matrix 
links $X^1,\ldots,$ $X^N$ in $C(\mathbb{I},M_{2n})$, which solve the problems
\[
 \Phi(x_j) \rightsquigarrow y_j\oplus y_j, \:\:\: 1\leq j\leq N,
\]
and satisfy the constraints
\[
\left\{
\begin{array}{l}
 \kappa(\Phi(x_j))=x_j,\\
 \|\Phi(x_j)-x_j\oplus x_j\|\leq \varepsilon,\\
 \|X_t^j-y_j\oplus y_j\|\leq \varepsilon,
\end{array}
\right.
\]
for each $1\leq j,k\leq N$ and each $t\in \mathbb{I}$. Moreover, $\ell_{\|\cdot\|}(X_t^j)\leq \varepsilon, \:\:\: 1\leq j\leq N$.
\end{theorem}
\begin{proof}
  By L.\ref{Joint_spectral_variation_inequality_2} we will have that given $\varepsilon>0$, 
 there are $0<\delta\leq\nu= \frac{\varepsilon}{2\pi}$ and an isospectral approximant $\Psi:=\mathrm{Ad}[W^*]$ (with $W\in\U{n}$) for $\NS{x}{N}$ 
 and $\NS{y}{N}$ such that, $\max\{\|x_j-\Psi(x_j)\|,\|y_j-\Psi(x_j)\|\}\leq \nu$. By setting 
 $\Phi:=(\Psi^{\dagger})^{[2]}\circ \imath_2\circ \Psi$, by D.\ref{Z2_dilations}, D.\ref{Std_dilations} and R.\ref{local_matching_remark} it can be seen that 
 $\Phi:M_n\to M_{2n}$ is a $C^*$-homomorphism such that $\|\Phi(x_j)-\imath_2(x_j)\|=\|\Phi(x_j)-x_j\oplus x_j\|\leq \varepsilon$, for each 
 $1\leq j\leq N$.

 Since $(\Psi^{\dagger})^{[2]}:=\mathrm{Ad}[\hat{W}_s]$ with $\hat{W}_s:=(\Sigma_2\otimes \I_n)(W^*\oplus W)$ and 
 since $\hat{W}_s\in \U{2n}\cap \mathbb{H}(2n)$, we will have that $\hat{W}_s$ can be represented as 
 $\hat{W}_s=e^{i\frac{\pi}{2}(\hat{W}_s-\I_{2n})}$ for any $n\geq 1$. If we set $\tilde{X}_j:=\Psi^{[s]}(\imath_2({x}_j))$, 
 $1\leq j\leq N$, we also have that there is a unitary path 
 $\{\mathcal{W}_t\}_{t\in\mathbb{I}}\subset M_{2n}$ with $\mathcal{W}_t:=e^{i\frac{\pi (1-t)}{2}(\hat{W}_s-\I_{2n})}$, which satisfies the 
 conditions  $\mathcal{W}_0=\hat{W}_s$, $\mathcal{W}_1=\I_{2n}$, together with the normed estimates
 \begin{eqnarray*}
  \|\mathcal{W}_t\tilde{X}_j-\tilde{X}_j\mathcal{W}_t\|&=&|\cos(\pi t/2)|\|\hat{W}_s\tilde{X}_j-\tilde{X}_j\hat{W}_s\|\\
                                      &\leq&\|\hat{W}_s\tilde{X}_j-\tilde{X}_j\hat{W}_s\| \leq \nu,
 \end{eqnarray*}
for each $1\leq j\leq N$ and each $0\leq t\leq 1$. Moreover, for each $1\leq j\leq N$ we have that the paths 
$\breve{X}^j_t:=\Ad{\mathcal{W}_t}{\tilde{X}_j}$ satisfy the normed estimates
\begin{eqnarray*}
 \ell_{\|\cdot\|}(\breve{X}_t^j)&=&\int_{\mathbb{I}}\|\partial_t \Ad{\mathcal{W}_t}{\tilde{X}_j})\|dt,\\
                                                  &=&\frac{\pi}{2}\|\hat{W}_s\tilde{X}_j-\tilde{X}_j\hat{W}_s\| \leq \nu.
\end{eqnarray*}

For each $1\leq j\leq N$, we can now use the flat paths $\bar{X}_t^j:=(1-t)\tilde{X}_j+t\imath_2(y_j)$ together with 
the previously described curved paths $\breve{X}^j$ to construct the solvent toral matrix links $X^1,\ldots,X^N\in C([0,1],M_{2n})$ 
we are looking for, and which can be defined by $X^j:=\breve{X}^j\circledast \bar{X}^j$ for each $1\leq j\leq N$, and we are done.
\end{proof}

\begin{remark}
 It can be seen that by using the technique implemented in the proof of T.\ref{lifted_local_connection_of_N_tuples_of_normal_contractions} one can 
 obtain lifted versions of C.\ref{local_connectivity_N_tuples_commuting_unitaries} and C.\ref{local_hermitian_interpolants}.
\end{remark}

\begin{remark}
\label{detection_methods}
 As a consequence of T.\ref{lifted_local_connection_of_N_tuples_of_normal_contractions} we can derive simple detection methods to identify families 
 of pairwise commuting matrices in $M_n$ that can be connected uniformly via piecewise analytic toral matrix links. The existence of these detection methods 
 raises some interesting questions for further studies.
\end{remark}

\begin{remark}
 We can interpret T.\ref{lifted_local_connection_of_N_tuples_of_normal_contractions} as an existence theorem of solutions to lifted 
 connectivity problems defined on matrix representations of the form
 \[
 \xymatrix{
& C_{\varepsilon}^*\langle\mathbb{Z}/2\times \mathbb{Z}\rangle \ar[r] & C^*(\hat{U}_s,\hat{V}) \ar[r] &
M_{2n} \ar[d] \\
C^*\langle\mathbb{F}_2\rangle \ar[r] \ar[ru] & C_{\delta}(\mathbb{T}^{2}) \ar[r] & C^*(U,V) \ar[u] \ar[r]  & M_n
},
\]
with $\hat{U}_s=(\Sigma_2\otimes \I_n)(U^*\oplus U)$ and $\hat{V}=V\oplus V$.
\end{remark}

Some further applications of T.\ref{lifted_local_connection_of_N_tuples_of_normal_contractions} to approximation of matrix 
words and norm behavior will be presented in \cite{Vides_matrix_words}.

\subsubsection{Matrix Klein Bottles: Local matrix deformations and special symmetries}
  Using T.\ref{lifted_local_connection_of_N_tuples_of_normal_contractions} we can solve all connectivity problems 
  (together with their softened versions) in $M_n$ that can be reduced to connectiviy problems of the form 
  $x\rightsquigarrow_{\varepsilon} x^*$ in $\N{n}(\disk)$, with $x^*=TxT$ and $T^2=\I_n$.

 \begin{remark}
 For each $\varepsilon\in [0,2]$, we can use the previously described symmetries and $\mathscr{D}_{\mathbb{T}}$ to 
 interpret $\bigcup_{x\in M_n}\{x\rightsquigarrow_\varepsilon x^*\}$ as matrix analogies of the {\em Klein bottle}.
 \end{remark}
 
  By a {\em softened matrix Klein bottle} we mean that the symmetries are softened, in particular we can consider 
  the connectiviy problems $x\rightsquigarrow_{\varepsilon} x^*$ and $y\rightsquigarrow_{\varepsilon} y^*$ in 
  $\N{n}(\disk)$ subject to 
  the normed constraints $\|xy-yx\|\leq \delta$, $\|x^*-TxT\|,\|xT-Ty\|\leq \delta$ 
  and $T^2=\I_n$. The details regarding to the solvability of these local connectivity problems will be addressed in 
  future communications.

\subsection{$C^0$ uniform local connectivity of pairs of unitaries and piecewise analytic approximants}

\label{C_0_technique}

The technique presented in this section can be used to solve local connectivity problems 
in matrix representations of the form $C_{\varepsilon}(\TT[2])\to M_n \leftarrow C(\TT[2])$ uniformly via $C^0$-unitary paths.

Suppose $U_{t}$ and $V_{t}$ are unitary matrices in $\mathbf{M}_{n}(\mathbb{C})$
for $t=0$ and $t=1$ and we define 
\begin{equation}
U_{t}=U_{0}e^{t\ln\left(U_{0}^{*}U_{1}\right)}\label{eq:U_t}
\end{equation}
 and 
\begin{equation}
V_{t}=V_{0}e^{t\ln\left(V_{0}^{*}V_{1}\right)}.\label{eq:V_t}
\end{equation}
For $t=0$ or $t=1$ the $C^{*}$-algebra generated by $U_{t}$ and
$V_{t}$ is abelian, so select a MASA $C_{t}\cong\mathbb{C}^{n}$
in each case. Let
\[
A(C_{0},C_{1})=\left\{ \left.X\in C\left([0,1],\mathbf{M}_{n}(\mathbb{C})\right)\right|X(0)\in C_{0}\mbox{ and }X(1)\in C_{1}\right\} .
\]

\begin{lemma}
\label{lem:tsr_one}The $C^{*}$-algebra $A(C_{0},C_{1})$ has stable
rank one.\end{lemma}
\begin{proof}
Starting with $X$ continuous with $X(t)$ in $C_{t}$ at the endpoints,
we can adjust this by a small amount, leaving the endpoints in $C_{t}$,
to get $X$ piece-wise linear, with the endpoints of every linear
segment having no spectral multiplicity and being invertible. Using
Kato's theory of analytic paths, we can get a piece-wise continuous
unitary $U_{t}$ and piece-wise analytic scalar paths $\lambda_{n}(t)$
so that the new path $Y\approx X$ satisfies
\[
Y(t)=U_{t}\left[\begin{array}{ccc}
\lambda_{1}(t)\\
 & \ddots\\
 &  & \lambda_{n}(t)
\end{array}\right]U_{t}^{*}.
\]
 There may be finitely may places where $Y(t)$ is not invertible.
These places will be in the interior of the segment so in an open
interval where $U_{t}$ is continuous. A small deformation of some
of the $\lambda_{j}$ will take the path through invertibles. We have
not moved the endpoints in the second adjustment so the constructed
element is in $A(C_{0},C_{1})$ and close to $X$.\end{proof}
\begin{lemma}
The endpoint-restriction map $\rho:A(C_{0},C_{1})\rightarrow C_{0}\oplus C_{1}$
induces an injection on $K_{0}$.\end{lemma}
\begin{proof}
The kernel of $\rho$ is $C\left([0,1],\mathbf{M}_{n}(\mathbb{C})\right)$
which has trivial $K_{0}$-group. So this result follows from the
exactness of the usual six-term sequence in $K$-theory.\end{proof}
\begin{lemma}
\label{lem:Bott_trivial} Given unitaries $U$ and $V$ in $A(C_{0},C_{1})$,
with $\left\Vert \left[U,V\right]\right\Vert \nu_0$ as if D.\ref{Loring_Bott_index_definition} (so
the Bott index makes sense), $\mathrm{Bott}(U,V)$ is the trivial
element of $K_{0}\left(A(C_{0},C_{1})\right)$. \end{lemma}
\begin{proof}
By the previous lemma, we need only calculate $\mathrm{Bott}(\rho(U),\rho(V))$.
These unitaries are in a commutative $C^{*}$-algebra so they have
trivial Bott index.\end{proof}
\begin{theorem}
\label{C_0_unitary_interpolants}
Given $\epsilon>0$, there exists $\delta>0$ so that for all $n$,
given unitary matrices $U_{0},$ $U_{1},$ $V_{0},$ $V_{1}$ in $\mathbf{M}_{n}(\mathbb{C})$
with $U_{0}V_{0}=V_{0}U_{0}$,  $U_{1}V_{1}=V_{1}U_{1}$, $\left\Vert U_{0}-U_{1}\right\Vert \leq\delta$
and $\left\Vert V_{0}-V_{1}\right\Vert \leq\delta$, then there exists
continuous paths $U_{t}$ and $V_{t}$ between the given pairs of
unitaries with each $U_{t}$ and $V_{t}$ unitary, and with $U_{t}V_{t}=V_{t}U_{t}$,
$\left\Vert U_{t}-U_{0}\right\Vert \leq\epsilon$ and $\left\Vert V_{t}-V_{0}\right\Vert \leq\epsilon$
for all $t$.\end{theorem}
\begin{proof}
The paths $U_{t}$ and $V_{t}$ defined in equations~\ref{eq:U_t}
and \ref{eq:V_t} will be almost commuting unitary elements of $A(C_{0},C_{1})$.
By Lemma~\ref{lem:tsr_one} we may apply \cite[Theorem 8.1.1]{Eilers_anticommutation}
regarding approximating in $A(C_{0},C_{1})$ by commuting unitaries.
Lemma~\ref{lem:Bott_trivial} tells us there is no invariant to worry
about, so we can find $A_{t}$ and $B_{t}$ close of $U_{t}$ and
$V_{t}$ that are commuting continuous paths of unitaries with $A_{t}$
and $B_{t}$ in $C_{t}$ for $t=0,1$. The unitary elements in the
commutative $C_{t}$ are locally connected, so we can find a short
path from $U_{0}$ and $V_{0}$ to $A_{0}$ and $B_{0}$, and likewise
at the other end. Concatenating, we get a paths of commuting unitary
matrices from $U_{0}$ and $V_{0}$ to $U_{1}$ and $V_{1}$ so that
at every point we are close to some pair $(U_{t},V_{t})$. These then
are all close to $U_{0}$ and $V_{0}$.
\end{proof}

By combining T.\ref{lifted_local_connection_of_N_tuples_of_normal_contractions}, C.\ref{local_connectivity_N_tuples_commuting_unitaries} 
and T.\ref{C_0_unitary_interpolants} it can be seen that.

\begin{remark}[Piecewise analytic approximants of $C^0$ interpolants]
\label{analytic_approximants}
 Given $\epsilon>0$, there exists $\delta>0$ so that for all $n$,
given unitary matrices $U_{0},$ $U_{1},$ $V_{0},$ $V_{1}$ in $\mathbf{M}_{n}(\mathbb{C})$
with $U_{0}V_{0}=V_{0}U_{0}$,  $U_{1}V_{1}=V_{1}U_{1}$, $\left\Vert U_{0}-U_{1}\right\Vert \leq\delta$
and $\left\Vert V_{0}-V_{1}\right\Vert \leq\delta$, there exist continuous (interpolants) paths 
$U_{t}$ and $V_{t}$ in $M_{2n}$ which solve the problems $U_0\oplus U_0 \rightsquigarrow U_1\oplus U_1$
and $V_0\oplus V_0 \rightsquigarrow V_1\oplus V_1$ with each $U_{t}$ and $V_{t}$ unitary, and with 
$U_{t}V_{t}=V_{t}U_{t}$,
$\left\Vert U_{t}-U_{0}\oplus U_0\right\Vert \leq\epsilon$ and $\left\Vert V_{t}-V_{0}\oplus V_0\right\Vert \leq\epsilon$
for all $t$. There are also a C$^*$-homomorphism $\Psi:M_n\to M_{2n}$ such that 
\[
 \max\{\|\Psi(U_0)-U_1\oplus U_1\|,\|\Psi(U_0)-U_0\oplus U_0\|,\|\Psi(V_0)-V_1\oplus V_1\|,\|\Psi(V_0)-V_0\oplus V_0\|\}\leq \epsilon,
\]
and two piecewise analytic unitary pairwise commuting paths $\hat{U},\hat{V}\in C([0,1],M_{2n})$ which solve the 
problems $\Psi(U_0) \rightsquigarrow U_1\oplus U_1$, 
$\Psi(V_0) \rightsquigarrow V_1\oplus V_1$ with $\max\{\|\hat{U}_t-U_t\|,\|\hat{V}_t-V_t\|\}\leq \epsilon$ for 
each $0\leq t\leq 1$. Moreover, $\ell_{\|\cdot\|}(\hat{U}_t)\leq \epsilon$ and $\ell_{\|\cdot\|}(\hat{V}_t)\leq \epsilon$.
\end{remark}

\subsection{Jointly compressible matrix sets}\label{compressible_matrix_sets} Given $0<\delta\leq \varepsilon$, we can now consider an alternative approach to 
the local connectivity problem involving two $N$-sets of pairwise commuting normal matrix contractions $\NS{X}{N}$ and $\NS{Y}{N}$ such 
that $\|X_j-Y_j\|\leq\delta$ for 
each $1\leq j\leq N$. The approach that we will consider in this section consists of considering the existence 
of a normal contraction $\hat{X}$ such that $\NS{X}{N}\in C^*(\hat{X})$, and which also satisfies the constraint $\|\hat{X}-X_j\|\leq \varepsilon$ for some 
$1\leq j\leq N$. A matrix $\hat{X}$ which satisfies the previous conditions will be called a {\em \bf nearby generator} for 
$\NS{X}{N}$, it can be seen that for any $\delta\leq \nu\leq \varepsilon$ one can find a flat analytic path $\bar{X}\in C([0,1],\Minf)$ 
that performs the deformation $X_j\rightsquigarrow_{\nu}\hat{X}$, where $\hat{X}$ is a nearby generator for 
$\NS{X}{N}$.

Given any joint isospectral approximant $\Psi$ with respect to the families normal contractions described in the previous paragraph, along the lines 
of the program that we have used 
to derive the connectivity results 
T.\ref{local_connection_of_N_tuples_of_normal_contractions} and T.\ref{lifted_local_connection_of_N_tuples_of_normal_contractions}, we 
can use L.\ref{Joint_spectral_variation_inequality_2} to find a C$^*$-automorphism which solves the extension problem described by 
the diagram,
\begin{equation}
\xymatrix{
& C^*(\hat{X}) \ar@{-->}[d]^{\hat{\Psi}}\\
C^*(\NS{X}{N}) \ar@{^{(}->}[ru] \ar[r]_{\Psi} & C^*(\NS{Y}{N})
}
 \label{NJC_lifting}
\end{equation}
and satisfies the relations $\Psi(X_j)=\hat{\Psi}(X_j)$ for each $1\leq j\leq N$ together with the normed constraints
\[
\max\{\|\hat{\Psi}(\hat{X})-\hat{X}\|,\max_j\{\|\hat{\Psi}(X_j)-X_j\|,\|\hat{\Psi}(X_j)-Y_j\|\}\}\leq \varepsilon. 
\]
We refer to the C$^*$-automorphism $\hat{\Psi}$ in \ref{NJC_lifting} as a {\bf compression} of $\Psi$ or a {\bf compressive joint isospectral approximant} 
({\bf CJIA}) for the $N$-sets of normal contractions. Let us now consider a special type of inner C$^*$-automorphisms that can be described as follows.

\begin{definition}[Uniformly compressible JIA] Given $0<\delta\leq \varepsilon$ and two $N$-sets of pairwise commuting 
normal contractions $\NS{X}{N}$ and $\NS{Y}{N}$ in $\Minf$ such that $\|X_j-Y_j\|\leq \delta$, $1\leq j\leq N$, a joint isospectral approximant 
$\Psi$ of the $N$-sets is said to be {\bf uniformly compressible} if there are a nearby generator $\hat{X}$ for $\NS{X}{N}$, a compression $\hat{\Psi}:=\mathrm{Ad}[W]$ (with $W\in \U{\Minf}$) of $\Psi$ and a unitary 
$\hat{W}\in \hat{\Psi}(C^*(\hat{X}))$ such that $\|W-\hat{W}\|\leq \varepsilon$. We refer to the $2N$ normal contractions 
$\NS{X}{N}$ and $\NS{Y}{N}$ for which there exists a uniformly compressible JIA ({\bf UCJIA}) as uniformly jointly compressible ({\bf UJC}).
\end{definition}

\begin{lemma}[Local connectivity of {\bf UJC} matrix sets]
\label{UJC_connectivity}
 Given $\varepsilon>0$, there is $\delta>0$, such that for any two $N$-sets of {\bf UJC} pairwise commuting normal contractions 
 $\NS{X}{N}$ and $\NS{Y}{N}$ in $\Minf$ such that $\|X_j-Y_j\|\leq \delta$ for each $1\leq j\leq N$, we will have that there are 
 $N$ toral matrix links $\mathbf{X}^1,\ldots,\mathbf{X}^N\in C([0,1],\Minf)$ that solve the interpolation 
 problem $X_j \rightsquigarrow_{\varepsilon} Y_j$, for each $1\leq j\leq N$.
\end{lemma}
\begin{proof}
 Since the $N$-sets of pairwise commuting normal contractions $\NS{X}{N}$ and $\NS{Y}{N}$ are {\bf UJC}, we have that given 
 $0<\delta\leq \nu\leq \varepsilon/2<1$, there are a normal contraction $\hat{X}\in \Minf$ which commutes with each $X_j$ together with a {\bf UCJIA} 
 $\hat{\Psi}=\mathrm{Ad}[W]$ for some $W\in \U{\Minf}$ and a unitary $\hat{W}\in \hat{\Psi}(C^*(\hat{X}))$ such that 
 \begin{equation}
  \|\I-\hat{W}^*W\|=\|W-\hat{W}\|\leq \nu<1.
  \label{compresion_inequality_1}
 \end{equation}
 Let us set $Z:=\hat{W}^*W$, as a consequence of the inequality \ref{compresion_inequality_1} we will have that there is a hermitian matrix $-\I\leq H_{Z}\leq \I$ in 
 $\Minf$ such that $e^{\pi iH_Z}=Z$. By using \ref{compresion_inequality_1} again, it can be seen that we can now use the curved paths 
 $\breve{\mathbf{X}}^j:=\Ad{e^{-\pi i t H_Z}}{X_j}$ to solve the problems $X_j\rightsquigarrow_{\varepsilon/2} \hat{\Psi}(X_j)$, and then we can solve 
 the problems $\hat{\Psi}(X_j)\rightsquigarrow_{\nu} Y_j$ using the flat paths $\bar{\mathbf{X}}^j:=(1-t)\hat{\Psi}(X_j)+tY_j$. We can construct the solvent toral matrix 
 links by setting $\mathbf{X}^j:=\breve{\mathbf{X}}^j\circledast \bar{\mathbf{X}}^j$ for each $1\leq j\leq N$. This completes the proof.
\end{proof}

\section{Hints and Future Directions}

The detection matrix representations of universal C$^*$-algebras that can be connected uniformly via piecewise analytic 
paths induces interesting problems which are topological/K-theoretical and computational in nature. Motivated by the $C^0$-connectivity technique, 
we consider that the use of T.\ref{local_connection_of_N_tuples_of_normal_contractions}, C.\ref{local_connectivity_N_tuples_commuting_unitaries} and 
T.\ref{lifted_local_connection_of_N_tuples_of_normal_contractions} and L.\ref{UJC_connectivity} to study local matrix connectivity in C$^*$-representations of the form 
\[
\xymatrix{
 C(\TT[N]) \ar@/^1pc/[rr] \ar[r] & M_{2n} & M_{n} \ar[l] & C_{\varepsilon}(\TT[2]) \ar[l] \ar@/^1pc/[ll]
}
\]
will present interesting challenges and questions 
that will be the subject of future study. In particular we are interested in the 
application of T.\ref{lifted_local_connection_of_N_tuples_of_normal_contractions}, C.\ref{local_connectivity_N_tuples_commuting_unitaries} and 
L.\ref{UJC_connectivity} to the study of the question. Is 
$C^*\langle\mathbb{F}_2\times \mathbb{F}_2\rangle$ RFD? (This is equivalent to 
{\bf Connes's embedding problem}.)

A better understanding 
of the geometric and approximate combinatorial nature of toroidal matrix links would provide a mutually benefitial interaction between matrix flows 
in the sense of Brockett \cite{Brockett_matching} and Chu \cite{Chu_num_lin}, topologically controlled linear algebra in the sense of 
Freedman and Press \cite{CLA_Freedman} and matrix geometric deformations in the sense of Loring \cite{Loring_deformations}. This also 
may provide some novel generic numerical methods to study and compute normal matrix compressions, sparse representations and 
dimensionality reduction of large scale matrices. Using a similar approach we plan to use T.\ref{lifted_local_connection_of_N_tuples_of_normal_contractions} and 
L.\ref{UJC_connectivity} to answer some questions in topologically controlled linear algebra in the sense of \cite{CLA_Freedman}, 
raised by M. H. Freedman.

The construction and generalization of detection methods like the 
ones mentioned in the remark R.\ref{detection_methods} of theorem T.\ref{lifted_local_connection_of_N_tuples_of_normal_contractions} 
together with their implications on inverse spectral problems, will be the subject of future studies. In particular we will use 
toral matrix links to study the local deformation properties of matrix representations of the form 
$C_{\varepsilon}(\TT) \rtimes_{\alpha} \ZZ/2\to M_n$ (where $\alpha$ denotes the standard flip) via {\em softened 
matrix Klein bottles}. These problems are related to spectral decomposition problems with {\em spectral symmetry} in quantum theory and to {\em deformation theory} 
for C$^*$-algebras in the sense of Loring \cite{Loring_deformations}. We will also use toroidal matrix links to study the local 
connectivity of some Soft group C$^*$-algebras in the sense of Farsi \cite{Soft_Algebras_Farsi}.

Some generalizations of T.\ref{lifted_local_connection_of_N_tuples_of_normal_contractions} and particular applications of 
L.\ref{UJC_connectivity} to the study of matrix equations on words will also be the subject of future study. In particular, 
the combination of toroidal matrix links with some matrix lifting techniques along the same lines of the proof of 
T.\ref{lifted_local_connection_of_N_tuples_of_normal_contractions} combined with L.\ref{UJC_connectivity}, seem also promising on the 
solvability of some conjectures studied numerically on \cite{Norms_of_Commutators}.

\section{Acknowledgement}
Both authors are very grateful with the Erwin Schr{\"o}dinger Institute for Mathematical Physics of the University of Vienna, for the outstanding hospitality during our visit to participate in the research program on Topological phases of quantum matter in 
August of 2014.  Much of the research reported in this document was carried out while we were visiting the Institute. 
The second author wants to thank Moody T. Chu for his warm hospitality during his visit to the Department of Mathematics at North 
Carolina State, for precise comments, challenging questions, and for sharing interesting conjectures and problems with him. He also wants 
to thank Alexandru Chirvasitu for several interesting questions and comments that have been very helpful for the preparation of this document.

This work was partially supported by a grant from the Simons Foundation (208723 to Loring).

\end{document}